\documentclass{amsart}
\newtheorem{theorem}{Theorem}

\newtheorem{prop}{Proposition}

\title{Some Determinantal Identities }
\author{Milan Janji\'c }
\address{Department of Mathematics and Informatics\\
 University of Banja Luka \\
Republic of Srpska, BA}
\begin{document}
\maketitle

\begin{abstract}
Some applications of a result, which is proved in \cite{jan}, is considered.
We first prove three determinantal identities concerning the binomial coefficient and Stirling numbers of the first and the second kind. We  also easily  obtain the inverse of the Vandermonde matrix. Then we derive a recurrence formula for sums of powers, which is similar to the well-known Newton identity.  In the last section, we consider some sequences given by a homogenous linear recurrence equation.
  A determinantal identity for the Fibonacci numbers of higher  order is proved.
  We finish with an expression of the generalized Vandermonde determinant in terms of the standard Vandermonde determinant and elementary symmetric polynomials.
\end{abstract}


\section{Introduction}
We restate the result proved in \cite{jan}.

Let $n$ and $r$ be  positive integers. We consider the following
$n+r-1$ by $r$ matrix:
\begin{equation}\label{mpa}P=\begin{pmatrix}
p_{1,1}&p_{1,2}&\cdots&p_{1,r-1}&p_{1,r}\\
p_{2,1}&p_{2,2}&\cdots&p_{2,r-1}&p_{2,r}\\
\vdots&\vdots&\cdots&\vdots&\vdots\\
p_{n,1}&p_{n,2}&\cdots&p_{n,r-1}&p_{n,r}\\
-1&p_{n+1,2}&\cdots&p_{n+1,r-1}&p_{n+1,r}\\
0&-1&\cdots&p_{n+2,r-1}&p_{n+2,r}\\
\vdots&\vdots&\cdots&\vdots&\vdots\\
0&0&\cdots&p_{n+r-2,r-1}&p_{n+r-2,r}\\
0&0&\cdots&-1&p_{n+r-1,r}
\end{pmatrix}.\end{equation}

We connect matrix (\ref{mpa}) with a recursively given sequence of
vector-columns in the following way: Let $A=(A_1|A_2|\ldots|A_n)$ be a square matrix of
order $n.$  Here, $A_1,\ldots,A_n$ are columns of $A.$   We define a block matrix
$A_r=[A|A_{n+1}|\cdots|A_{n+r}]$ of $n$ rows and $n+r$ columns in the following way:
 \begin{equation}\label{r1} A_{n+j}=\sum_{i=1}^{n+j-1}p_{i,j}A_i,\;(j=1,2,\ldots,r).\end{equation}

For a  sequence  $1\leq j_1<j_2<\cdots<j_r<n+r$ of positive
integers, we
 let $M=M(\widehat{j_1},\widehat{j_2},\ldots,\widehat{j_r})$ denote  the  minor of $A_r$ of order $n,$  obtained  by deleting  columns  $j_1,j_2,\ldots,j_r$ of $A_r.$
 We shall also write $M(\widehat{j_1},\ldots,\widehat{j_r},A_{i_1},A_{i_2},\ldots,A_{i_k})$ if we want to stress that $M$ contains  $i_1,i_2,\ldots,i_k$ columns of $A_r.$

Note that the last column of $A_r$ cannot be deleted.

The sign ${\rm sgn}(M)$ of $M$  is defined as
      \[{\rm sgn}(M)=(-1)^{nr+j_1+j_2+\cdots+j_r+\frac{(r-1)r}{2}}.\]
We let  $Q=Q(j_1,\ldots,j_r)$ denote the submatrix of order $r,$
laying in $j_1,j_2,\ldots,j_r$ rows of $P.$

\begin{theorem}\label{th1} Let $1\leq j_1<\cdots<j_r<r+n$ be a sequence of positive integers.
Then,
\begin{equation}\label{ttt}M(\widehat{j_1},\ldots,\widehat{j_r})={\rm sgn}(M)\cdot \det Q\cdot \det A.\end{equation}
\end{theorem}
In particular, for $r=1$ we have
\begin{theorem}\label{t1}
Let $A=(A_1,A_2,\ldots,A_n)$ be a matrix  of order $n,$ and let $p_1,p_2,\ldots,p_n$ be arbitrary elements of $F.$ If
\begin{equation}A_{n+1}=\sum_{i=1}^np_iA_i,\end{equation} then, for $j=1,2,\ldots,n,$ we have \[\det (A_1,\ldots,\widehat A_j,\ldots,A_{n+1})=(-1)^{n-j}p_j\det A.\]
\end{theorem}
\section{Some Determinantal Identities}
We consider two function $f,g:\mathbb Z^+\times \mathbb Z^+\to \mathbb Z^+,$
such that $f(i,j)=0,$ if $i>j,$ and $g(i,j)=0,$ if $i<j.$
Let $f,g$ satisfy the  following recurrence:
\begin{equation}\label{rro}f(n+1,k+1)=\sum_{i=0}^ng(n,i)\cdot f(i,k).\end{equation}
We want to express $g$ in terms of $f.$

Define an upper-triangular matrix $A=(a_{ij})=(A_0,A_1,\ldots,A_n),$ of order $n+1,$
where $A_0,A_1,\ldots,A_n$ are the vector-columns of $A,$ such that
$a_{ij}=f(j,i),\;(i,j=0,\ldots,n).$

We next define $A_{n+1}=(a_{0,n+1},a_{1,n+1},\ldots,a_{n,n+1})^T,$ such that

\[a_{i,n+1}=\sum_{t=0}^{n}g(n,t)a_{it},\;(i=0,1,\ldots,n).\]
It follows that
\[a_{i,n+1}=\sum_{t=0}^{n}g(n,t)f(t,i)=f(n+1,i+1),\;(i=0,1,\ldots,n).\]
Hence,
\[A_{n+1}=(f(n+1,1),f(n+1,2),\ldots,f(n+1,n+1))^T.\]
We thus obtain
\[A_{n+1}=\sum_{j=0}^ng(n,j)A_j.\]

We conclude that matrix $A$ satisfies the conditions of Theorem \ref{t1}.
  Taking into account that $A$ is an upper-triangular matrix, we obtain the following identity:
\[\det (A_0,\ldots,\widehat A_j,\ldots,A_{n+1})=(-1)^{n-j+1}g(n,j)
\prod_{i=0}^na_{ii},\;(j=0,1,\ldots,n).\]

 Note that the matrix $(A_0,\ldots,\widehat A_j,\ldots,A_{n+1})$ is a quasi-diagonal block matrix
of the form \[(A_0,\ldots,\widehat A_j,\ldots,A_{n+1})=\text{diag}(A_{11},A_{22}).\]
The matrix $A_{11}$ is an upper-triangular matrix, which determinant equals $\prod_{k=0}^{j-2}f(j,j).$ The matrix $A_{22}$ is an upper-Hesenberg matrix of order $n-j+2,$ which has the  form:\[A_{22}=\begin{pmatrix}f(j,j-1)&f(j+1,j-1)&\cdots&f(n,j-1)&f(n+1,j)\\
f(j,j)&f(j+1,j)&\cdots&f(n,j)&f(n+j,j+1)\\
0&f(j+1,j+1)&\cdots&f(n,j+1)&f(n+1,j+2)\\
\vdots&\vdots&\cdots&\vdots&\vdots\\
0&0&\cdots&f(n,n)&f(n+1,n+1)
\end{pmatrix}.\]
Assuming that $f(i,i)\not=0$ we obtain
\begin{equation}\label{id}g(n,j-1)=
(-1)^{n-j+1}\prod_{t=j-1}^nf(t,t)\det A_{22}.\end{equation}
The following are the well-known recurrences for the binomial coefficients, Stirling numbers of the first, and Stirling numbers of the second kind:
\[{n+1\choose k+1}=\sum_{i=k}^{n}{i\choose k},\;\left[n+1\atop k+1\right]=\sum_{i=0}^{n}(-1)^{n-i}\frac{n!}{i!}\left[i\atop k\right],\;\left\{n+1\atop
k+1\right\}=\sum_{i=0}^{n}{n\choose i}\left\{i\atop k\right\}.\]
In the view of (\ref{id}) we have 
\begin{prop}
\begin{enumerate}
\item
The binomial coefficients satisfy the following identity:
\[1=(-1)^{n-j+1}\begin{vmatrix}{j\choose j-1}&{j+1\choose j-1}&\cdots&{n-1\choose j-1}&{n\choose j}\\
{j\choose j}&{j+1\choose j}&\cdots&{n-1\choose j}&{n\choose j+1}\\
\vdots&\vdots&\cdots&\vdots&\vdots\\
0&0&\cdots&{n-1\choose n-1}&{n\choose n}
\end{vmatrix} .\]
\item For the Stirling numbers of the first kind we have
\[\frac{n!}{(j-1)!}=\begin{vmatrix}\left[j\atop j-1\right]&\left[j+1\atop j-1\right]&\cdots&\left[n-1\atop j-1\right]&\left[n\atop j\right]\\
\left[j\atop j\right]&\left[j+1\atop j\right]&\cdots&\left[n-1\atop j\right]&\left[n\atop j+1\right]\\
\vdots&\vdots&\cdots&\vdots\\
0&0&\cdots&\left[n-1\atop n-1\right]&\left[n\atop n\right]
\end{vmatrix}.\]
\item For the Stirling numbers of the second kind we have
\[{n\choose j-1}=(-1)^{n-j+1}\begin{vmatrix}\left\{j\atop j-1\right\}&\left\{j+1\atop j-1\right\}&\cdots&\left\{n-1\atop j-1\right\}&\left\{n\atop j\right\}\\
\left\{j\atop j\right\}&\left\{j+1\atop j\right\}&\cdots&\left\{n-1\atop j\right\}&\left\{n\atop j+1\right\}\\
\vdots&\vdots&\cdots&\vdots\\
0&0&\cdots&\left\{n-1\atop n-1\right\}&\left\{n\atop n\right\}
\end{vmatrix}.\]
\end{enumerate}
\end{prop}

\section{Inverse of Vandermonde Matrix}
Using Theorem \ref{t1}, we easily derive the inverse of a Vandermondove matrix.
 We let $e_k=e_k(x_1,x_2,\ldots,x_n),\;(k=0,1,2,\ldots,n)$  denote the elementary symmetric polynomials of $x_1,x_2,\ldots,x_{n}.$ We use the following notation: $e_k(\widehat{x_i})=e_k(x_1,\ldots,x_{i-1},x_{i+1},\ldots,x_n).$
As usual, by $V=V(x_1,\ldots,x_{n})$ will be denoted the Vandermonde matrix of order $n.$

\begin{prop}\label{p1} If $V=V(x_1,\ldots,x_n)$ is the Vandermonde matrix, and $M_{ij}$ is the minor of order $n-1,$ obtained by deleting the row $i$ and the column $j$ of $V,$ then
 \begin{equation}\label{min}M_{ij}=e_{n-j}(\widehat{x_i})\det V(\widehat{x_i}).\end{equation}
\end{prop}

\begin{proof} For a fixed $i$ the minors $M_{ij},\;(j=1,2,\ldots,n)$  are the minors of order $n-1$  of the
matrix
\[M_i=  \begin{pmatrix}1&x_1&\cdots&x^{n-1}_1\\1&x_2&\cdots&x^{n-1}_2\\\vdots&\vdots&\cdots&\vdots\\
\widehat{1}&\widehat{x_i}&\cdots&\widehat{x^{n-1}_i}\\\vdots&\vdots&\cdots&\vdots\\
1&x_n&\cdots&x^{n-1}_{n}\end{pmatrix}.\]
Define the polynomial
\[f_i(x)=(x-x_1)(x-x_2)\cdots\widehat{(x-x_i)}\cdots(x-x_{n}),\;(i=1,2,\ldots,n).\]
Expanding the product on the right side of this equation one obtains
\[f_i(x)=x^{n-1}-e_1(\widehat{x_i})x^{n-2}+\cdots+(-1)^{n-1}e_{n-1}(\widehat{x_i}).
\]
It follows that \[x_k^{n-1}=e_1(\widehat{x_i})x_k^{n-2}-\cdots+(-1)^{n}e_{n-1}(\widehat{x_i}),\;(k\not=i).
\]
We conclude that, for the last column $C_{n}$ of $M_i$ we have
\[C_{n}=\sum_{j=1}^{n}(-1)^{n-j+1}e_{n-j}(\widehat{x_i})C_{j}.\]

Applying Theorem \ref{t1} we obtain (\ref{min}).

It is well-known that the following holds
\[\det V=\prod_{1\leq i<j\leq n}(x_j-x_i).\]
It implies that $V$ is invertible if and only if $x_i\not=x_j,\;(i\not=j).$
Suppose that $x_i\not=x_j,\;(i\not=j),$ and denote by $W=(w_{ij})=V^{-1}.$
\end{proof}
\begin{prop}\label{p2}
For each $i,j=1,2,\ldots, n$ we have
\[w_{ij}=\frac{(-1)^{n-i}e_{n-i}(\widehat{x_j})}{\prod_{k\not=j}(x_j-x_k)}.\]
\end{prop}
\begin{proof}
We have
\[w_{ij}=\frac{(-1)^{i+j}M_{ji}}{\det V}=\frac{(-1)^{i+j}e_{n-i}(\widehat{x_j})
\det V(\widehat{x_j})}{\det V}.\]
Since
\[\frac{\det V(\widehat{x_j})}{\det V}=\frac{(-1)^{n-j}}{\prod_{k\not=j}(x_j-x_k)},\]
we finally have
\[w_{ij}=\frac{(-1)^{n-i}e_{n-i}(\widehat{x_j})}{\prod_{k\not=j}(x_j-x_k)},\;(i,j=1,2,\ldots,n),\]
and the theorem is proved.
\end{proof}

\section{Sums of Powers}

For  a nonnegative integer $k,$  we denote $s_k=x_1^k+\cdots+x_n^k.$
\begin{prop} For a nonnegative integer $k$, the following recurrence holds
\begin{equation}\label{ss}s_{n+k}=\sum_{j=1}^{n}(-1)^{n+j}s_{k+j-1}e_{n-j+1}.\end{equation}
\end{prop}
\begin{proof}
Consider the following matrix of order $n+1.$
\[A=\begin{pmatrix}1&x_1&\cdots&x_1^{n}\\\vdots&\vdots&\vdots&\vdots\\
1&x_{i}&\cdots&x_{i}^{n}\\
\vdots&\vdots&\vdots&\vdots\\
1&x_n&\cdots&x_n^{n}\\
s_k&s_{k+1}&\cdots&s_{k+n}\\
\end{pmatrix}.\]
If we let $R_1,\ldots,R_n,R_{n+1}$ denote the rows of $A,$ then
\[R_{n+1}=\sum_{j=1}^{n}x_i^kR_i,\]
which means that the last row of $A$ is a linear combination of the remaining rows.
We conclude that $\det A=0.$
On the other hand, expanding $\det A$ across the last row, we obtain
\[\det A=\sum_{j=1}^{n+1}(-1)^{n+1+j}s_{m+j-1}M_{n+1,j},\]
where
\[M_{n+1,j}=\det\begin{pmatrix}1&x_1&\cdots&\widehat{x_1^{j-1}}&\cdots&x_1^n\\
\vdots&\vdots&\cdots&\vdots&\cdots&\vdots\\
1&x_n&\cdots&\widehat{x_n^{j-1}}&\cdots&x_n^n
\end{pmatrix}.\]
Since
\begin{equation}\label{eovi}x_m^{n}=\sum_{j=0}^{n-1}(-1)^{n-j+1}e_{n-j}x_m^j,\;(m=1,2,\ldots,n),\end{equation}
where $e_i=e_i(x_1,\ldots,x_n),$ we may apply Theorem \ref{t1}.

It follows that  $M_{i,j}=e_{n-j+1}\det V(x_1,\ldots,x_n),$
and the assertion is proved, under the conditions $\det V(x_1,x_2,\ldots,x_n)\not=0.$
\end{proof}
\section{Homogenous Linear Recurrence}
We consider the case when $p_{ij}=0,\;(j>i),$ in (\ref{mpa}). Then,
\begin{equation}\label{mpah}P=\begin{pmatrix}
p_{1,1}&0&\cdots&0&0\\
p_{2,1}&p_{2,2}&\cdots&0&0\\
\vdots&\vdots&\cdots&\vdots&\vdots\\
p_{n,1}&p_{n,2}&\cdots&p_{n,r-1}&p_{n,r}\\
-1&p_{n+1,2}&\cdots&p_{n+1,r-1}&p_{n+1,r}\\
0&-1&\cdots&p_{n+2,r-1}&p_{n+2,r}\\
\vdots&\vdots&\cdots&\vdots&\vdots\\
0&0&\cdots&p_{n+r-2,r-1}&p_{n+r-2,r}\\
0&0&\cdots&-1&p_{n+r-1,r}
\end{pmatrix}.\end{equation}
In this case, we have
\begin{equation}\label{r1h} A_{n+j}=\sum_{i=j}^{n+j-1}p_{i,j}A_i,\;(j=1,2,\ldots,r),\end{equation}
which is the homogenous recurrence equation of order $n,$ with the initial conditions $(A_1,A_2,\ldots,A_n).$ If, additionally, in (\ref{mpah}), all $p_{i,j}=1,$ then we have the recurrence for Fibonacci $n$-step numbers. We let $F^{(n,i)}_k,\;(i=1,2,\ldots,n,\;k=1,2,\ldots)$ denote Fibonacci $n$-step numbers, which the initial conditions are given by the $i$th row of $A.$ Taking $r=n,\;j_i=i,\;(i=1,2,\ldots,n),$ as a consequence of Theorem \ref{th1} we obtain
\begin{prop} The following identity is true:
\[\begin{vmatrix}F^{n,1}_{n+1}&F^{n,1}_{n+2}&\cdots&F^{n,1}_{2n}\\
\vdots&\vdots&\cdots&\vdots\\
F^{n,i}_{n+1}&F^{n,i}_{n+2}&\cdots&F^{n,i}_{2n}\\
\vdots&\vdots&\cdots&\vdots\\
F^{n,n}_{n+1}&F^{n,n}_{n+2}&\cdots&F^{n,n}_{2n}
\end{vmatrix}
=\det A.\]
\end{prop}

At the end we derive an explicit formula for the Generalized Vandermonde matrices.

For the sequence $0<k_1<\cdots<k_n$ of integers, the  matrix   \begin{equation}\label{vg}V(k_1,\ldots,k_n)=        \begin{pmatrix}
          x_1^{k_1-1} & x_1^{k_2-1} & x_1^{k_3-1} &\cdots&x_1^{k_n-1} \\
          x_2^{k_1-1} & x_2^{k_2-1} & x_2^{k_3-1} &\cdots&x_2^{k_n-1} \\
          \vdots &\vdots & \vdots &\vdots \\
           x_n^{k_1-1} & x_n^{k_2-1} & x_n^{k_3-1} &\cdots&x_n^{k_n-1} \\
        \end{pmatrix}
      \end{equation} is called the generalized Vandermonde matrix.

For each $k\geq 1,$ according to (\ref{eovi}), we have
\begin{equation}\label{jjj3}x_i^{n+k}=\sum_{j=1}^{n}(-1)^{n-j}e_{n-j+1}x_i^{k+j-1},
\;(i=1,2,\ldots,n).\end{equation}
We denote $\sigma_j=(-1)^{n-j}e_{n-j+1},\;(j=1,\ldots,n).$
Consider the following matrix
\begin{equation}\label{ma}\overline V=
        \begin{pmatrix}
          1 & x_1 & x_1^2 &\cdots&x_1^{n-1}& x_1^n&\cdots&x_1^{k_n-1}\\
         1 & x_2 & x_2^2 &\cdots&x_2^{n-1}& x_2^n&\cdots&x_2^{k_n-1}\\
          \vdots &\vdots & \vdots &\vdots &\vdots&\vdots\\
          1 & x_n & x_n^2 &\cdots&x_n^{n-1}& x_n^n&\cdots&x_n^{k_n-1}
        \end{pmatrix}
      .\end{equation}

We have
\begin{equation}\label{jjj3}
V_{n+k}=\sum_{j=1}^{n}(-1)^{n-j}e_{n-j+1}V_{k+j-1},\;(k=1,2,\ldots,k_n-n).\end{equation}

In the  view of (\ref{jjj3}), the corresponding  matrix  $P$ in (\ref{mpah}) is the $k_n$ by $k_n+1-n$ matrix of the form
\[P=\begin{pmatrix}
(-1)^{n-1}e_n&0&0&\cdots&0&0&0\\
(-1)^{n-2}e_{n-1}&(-1)^{n-1}e_n&0&\cdots&0&0&0\\
(-1)^{n-3}e_{n-2}&(-1)^{n-2}e_{n-1}&\sigma_0&\cdots&0&0&0\\
\vdots&\vdots&\vdots&\cdots&0&0&0\\
e_1&-e_2&e_3&\cdots&0&0&0\\
-1&e_1&-e_2&\cdots&0&0&0\\
0&-1&e_1&\cdots&0&0&0\\
\vdots&\vdots&\vdots&\cdots&\vdots&\vdots&\vdots\\
0&0&0&\cdots&e_1&-e_2&e_3\\
0&0&0&\cdots&-1&e_1&-e_2\\
0&0&0&\cdots&0&-1&e_1
\end{pmatrix}.\]
For ${\rm sgn}(M)$ one easily obtains that
\[{\rm sgn}(M)=(-1)^{\frac{n(n-1)}{2}+\sum_{i=1}^{n-1}k_i}.\]
Next, the corresponding matrix $Q$ is obtained  by deleting rows of $P$  indices of which are $k_1,k_2,\ldots,k_{n-1}.$
 Theorem \ref{th1} implies

 \begin{prop}\label{gvd} The following formula holds
\[ \det V(k_1,k_2,\ldots,k_n)={\rm sgn}(M)\cdot \det Q\cdot \det V(x_1,x_2,\ldots,x_n).\]
\end{prop}

Consider the particular case $k_i=i,\;(i=1,\ldots,n-1),\;(k_n=m>n).$
We  have ${\rm sgn}(M)=1,$ and $Q$ is the upper-Hessenberg matrix of order $m-n$:
\[Q=\begin{pmatrix}
e_1&-e_2&\cdots&*&*\\
-1&e_1&\cdots&*&*\\
0&-1&\cdots&*&*\\
\vdots&\vdots&\cdots&\vdots&\vdots\\
0&0&\cdots&e_1&*\\
0&0&\cdots&-1&e_1\end{pmatrix},\]
where $*$, depending on $m,$ has to be replaced with  either  $0$ or some of $\sigma$'s.
\begin{prop}For $m>n$ we have
\[\begin{vmatrix}1&x_1&\cdots&x_1^{n-2}&x_1^{m-1}\\
1&x_2&\cdots&x_2^{n-2}&x_2^{m-1}\\
\vdots&\vdots&\cdots&\vdots&\vdots\\
1&x_n&\cdots&x_n^{n-2}&x_n^{m-1}\\
\end{vmatrix}=\begin{vmatrix}
e_1&-e_2&\cdots&*&*\\
-1&e_1&\cdots&*&*\\
\vdots&\vdots&\cdots&\vdots&\vdots\\
0&0&\cdots&e_1&*\\
0&0&\cdots&-1&e_1\end{vmatrix}\cdot
\begin{vmatrix}1&x_1&\cdots&x_1^{n-2}&x_1^{n-1}\\
1&x_2&\cdots&x_2^{n-2}&x_2^{n-1}\\
\vdots&\vdots&\cdots&\vdots&\vdots\\
1&x_n&\cdots&x_n^{n-2}&x_n^{n-1}\\
\end{vmatrix}
.\]

\end{prop}

\end{document}